\documentclass{amsart}
\usepackage{amssymb}

\usepackage{graphicx}

\newtheorem{theorem}{Theorem}[section]
\newtheorem{lemma}[theorem]{Lemma}

\newtheorem{claim}[theorem]{Claim}
\newtheorem{question}[theorem]{Question}

\theoremstyle{definition}

\theoremstyle{remark}

\numberwithin{equation}{section}

\usepackage{caption} 
\begin{document}

\title[Hyperbolic L-space knots and their formal semigroups]{Hyperbolic L-space knots and their formal semigroups}


\author{Masakazu Teragaito}
\address{Department of Mathematics and Mathematics Education, Hiroshima University,
1-1-1 Kagamiyama, Higashi-hiroshima, Japan 739-8524.}
\email{teragai@hiroshima-u.ac.jp}
\thanks{The author  has been supported by
JSPS KAKENHI Grant Number 20K03587. }


\date{}


\commby{}

\begin{abstract}
For an L-space knot, the formal semigroup is defined from its Alexander polynomial.
It is not necessarily a semigroup.
That is, it may not be closed under addition.
There exists an infinite family of hyperbolic L-space knots whose formal semigroups
are semigroups generated by three elements.
In this paper, we give the first infinite family of hyperbolic L-space knots whose formal semigroups
are semigroups generated by five elements.
\end{abstract}

\maketitle


\section{Introduction}

A knot  in the $3$--sphere is called an \textit{L--space knot\/}  if it admits a Dehn surgery yielding an L--space.
An L--space $Y$ is a rational homology $3$-sphere with the simplest Heegaard Floer homology, that is,
$\mathrm{rank}\,\widehat{HF}(Y)=|H_1(Y;\mathbb{Z})|$.
Typical examples are the knots admitting lens space surgeries, such as torus knots.
There are several known constraints for L--space knots \cite{N, OS}.
Such a knot $K$ is fibered, and its Alexander polynomial $\Delta_K(t)$ has the form of
\begin{equation}\label{eq:lspaceknot}
\Delta_K (t)=1-t^{a_1}+t^{a_2}-\dots+t^{a_{2k}},
\end{equation}
where $0<a_1<a_2<\dots <a_{2k}$ for some $k$, and $a_{2k}$ equals to twice of the knot genus.
Also, it is known that $a_1=1$ by \cite{HW}.
In general, it seems that there remained little to be clarified about the distribution of $a_i$.
See \cite{Ta1,Ta2} for lens space surgery case.

Wang \cite{W} introduced the \textit{formal semigroup\/} $\mathcal{S}$ for an L--space knot $K$.
It is a set of nonnegative integers defined from the formal power series expansion
\[
\frac{\Delta_K(t)}{1-t}=\Sigma_{s\in \mathcal{S}}t^s \in \mathbb{Z}[[t]].
\]
The form of (\ref{eq:lspaceknot}) implies that $0\in \mathcal{S}$.
Hence, if a formal semigroup is a semigroup, then it is a monoid.
Nevertheless, we use the term  ``formal semigroup'' in deference to previous research.

Essentially, \cite{BL0,BL} discuss the same notion precedently.
In \cite{RR},  it is called the support of the Turaev torsion.
We remark that $\Delta_K(t)/(t-1)$ is usually called the Reidemeister--Milnor torsion or Turaev torsion in literatures.

For example, a torus knot of type $(2,7)$ has the Alexander polynomial
$1-t+t^2-t^3+t^4-t^5+t^6$.  Hence $\mathcal{S}=\{0,2,4\}\cup \mathbb{Z}_{\ge 6}$.
Another typical example is the $(-2,3,7)$--pretzel knot.
Its Alexander polynomial is $1-t+t^3-t^4+t^5-t^6+t^7-t^9+t^{10}$, so
$\mathcal{S}=\{0,3,5,7,8\}\cup \mathbb{Z}_{\ge 10}$.
(In this paper, we use the notation $\mathbb{Z}_{\ge m}$ for the set of integers bigger than or equal to $m$.)

In general, a formal semigroup is a subset of $\mathbb{Z}_{\ge 0}$, and 
we use the addition as the binary operation.
As seen from the above example, a formal semigroup is not necessarily closed under the addition.
It is an easy and known fact that a torus knot of type $(p,q) \ (p,q>0)$ has  the formal semigroup
$\langle p,q\rangle\ (=\{ap+bq\mid a,b\ge 0\})$, which
is a semigroup (see \cite[Example 1.10]{BL}).
Also, the formal semigroup of an iterated torus L--space knot is a semigroup \cite{W}.
However,  this is not the case for the $(-2,3,7)$--pretzel knot, because
$3\in \mathcal{S}$ but $6\not\in \mathcal{S}$.
More generally, it is straightforward to verify that any hyperbolic Montesinos L--space knot
has the formal semigroup which is not a semigroup. 
Because such a knot is known to be the $(-2,3,2n+1)$--pretzel knot with $n\ge 3$ by \cite{BMo},
and its Alexander polynomial given by \cite{H} immediately implies that
$3\in \mathcal{S}$ but $6\not\in \mathcal{S}$.
Also, we checked that most of hyperbolic Berge knots have formal semigroups which are not semigroups.
Hence it is not too much to say that  the formal semigroup of a hyperbolic L--space knot
is less likely to become an actual semigroup.

In \cite[Question 2.8]{W}, Wang asked if there exists  an L--space knot which is not an iterated torus knot and whose 
formal semigroup is a semigroup.
As explained in \cite{BK},
the author found two hyperbolic L--space knots \texttt{K8\_201} and \texttt{K9\_449} whose formal semigroups are actual semigroups.
Indeed, they are the only knots among $630$ hyperbolic L--space knots listed by Dunfield.
(This list is found in \cite{A}.  There were two unclassified knots in Dunfield's data, but they are confirmed
to be L--space knots by \cite{BKM}. The formal semigroups of these two are not semigroups.)
On the other hand, Baker and Kegel \cite{BK} show that \texttt{K9\_449} is the only knot
among Dunfield's list that is not the closure of a positive braid,
and try to generalize it to an infinite family of hyperbolic L--space knots $\{K_n\}$, where $K_1$ is \texttt{K9\_449}.
The author also found that their knots give the first infinite family of hyperbolic L--space knots
whose formal semigroups are semigroups.
More precisely, the formal semigroup of $K_n$ is $\langle 4,4n+2,4n+5\rangle$.
See \cite{BK} for more details.

For a finite set of positive integers $\{p_1,p_2,\dots, p_k\}$,
\[
\langle p_1,p_2,\dots, p_k\rangle=\{a_1p_1+a_2p_2+\dots+a_kp_k \mid a_i\in \mathbb{Z}_{\ge 0} \}
\]
is a semigroup under the addition. 
Since we include 0 as a coefficient, the identity element $0$ is excluded from generators.
For such a semigroup, the rank is defined to be the  minimal cardinality of a generating set.

Thus, the formal semigroup of the hyperbolic L--space knot $K_n$ in \cite{BK} is a semigroup of rank three.
On the other hand, the cabling formula of \cite{W} implies that
the semigroup of an iterated torus L--space knot can have arbitrarily high rank.
Then it is natural to ask the following.

\begin{question}
Does there exist a hyperbolic L--space knot whose formal semigroup is a semigroup with arbitrarily high rank?
\end{question}

The purpose of the present paper is to construct a new family of hyperbolic L--space knots
whose formal semigroups are semigroups of rank  $5$.

\begin{theorem}\label{thm:main}
There exists an infinite family of hyperbolic L--space knots whose formal semigroups
are semigroups of rank $5$.
\end{theorem}

The formal semigroup of an L--space knot is related to its knot Floer complex (see \cite{K}).
However, the meaning of closedness under addition in the formal semigroup seems to be missing.

\section{The family of hyperbolic L--space knots}

For any integer $n\ge 1$, let $\beta_n$ be the $6$--braid defined as
\[
\beta_n=(\sigma_3 \sigma_2\sigma_4 \sigma_1\sigma_3\sigma_5\sigma_2\sigma_4\sigma_3)^{2n+1}\sigma_3\sigma_2\sigma_1\sigma_1\sigma_2\sigma_3\sigma_2\sigma_1\sigma_2\sigma_2,
\]
where $\sigma_i$ is the standard generator in the $6$--strand braid group.
See Fig.~\ref{fig:braid}.
Let $K_n$ be the knot obtained as the closure of $\beta_n$.

\begin{figure}[tb]
\begin{center}
\includegraphics[scale=0.5]{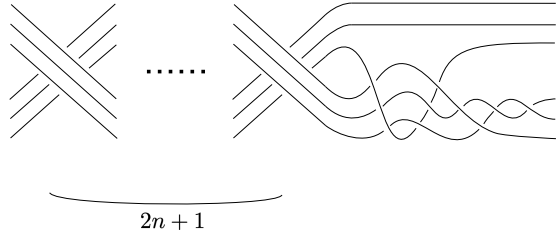}
\caption{The braid $\beta_n$. The knot $K_n$ is the closure of $\beta_n$.}\label{fig:braid}
\end{center}
\end{figure}

Since $\beta_n$ is positive, $K_n$ is fibered (\cite{S}), and its genus $g(K_n)$ is $9n+7$
as seen from an Euler characteristic calculation
\[
(\text{number of strand})-(\text{number of crossing})=
6-(9(2n+1)+10)=1-2g(K_n).
\]

We remark that $\beta_0$ can be defined, but $K_0$ is the $(2,11)$--cable of the trefoil.
Hence this is excluded from our interest.

Theorem \ref{thm:main} is the direct consequence of the next theorem.

\begin{theorem}\label{thm:main1}
For $n\ge 1$, let $K_n$ be the knot defined as above.
Then we have\textup{:}
\begin{itemize}
\item[\textup{(1)}]
$K_n$ is a hyperbolic L--space knot\textup{;} and
\item[\textup{(2)}]
its formal semigroup is a semigroup 
$\langle 6,6n+4,6n+8,12n+11,12n+15\rangle$.
\end{itemize}
\end{theorem}

The proof of Theorem \ref{thm:main1} is divided in the remaining sections.
In Section \ref{sec:montesinos},
we prove that $(18n+22)$--surgery on  $K_n$  yields an L--space by using the Montesinos trick.
In Section \ref{sec:alex}, we calculate the Alexander polynomial and the formal semigroup.
Finally, we prove that $K_n$ is hyperbolic in Section \ref{sec:hyp}.

\section{Montesinos trick and L--space surgery}\label{sec:montesinos}

In this section, we prove that each $K_n$ admits a Dehn surgery yielding an L--space
by using the Montesinos trick.
The Montesinos trick is the standard tool introduced by \cite{Mon}.
For a strongly invertible link $L$ in $S^3$,
the resulting manifold by Dehn surgery on $L$ is described
as the double branched cover of some link $\ell$.
The surgery corresponds to a tangle replacement.
For details, see \cite{MT,Wa}.

For $K_n$, Fig.~\ref{fig:surgery} shows a surgery diagram of $K\cup C_1\cup C_2$,
where performing  $-1/n$--surgery on $C_1$ and $1/2n$--surgery on $C_2$ changes $K$ to $K_n$.
This diagram can be changed into a strongly invertible position as illustrated in Fig.~\ref{fig:mont1}.
We remark that $r$--surgery on $K$ corresponds to $(18n+r)$--surgery on $K_n$.

\begin{figure}[tb]
\begin{center}
\includegraphics[scale=0.4]{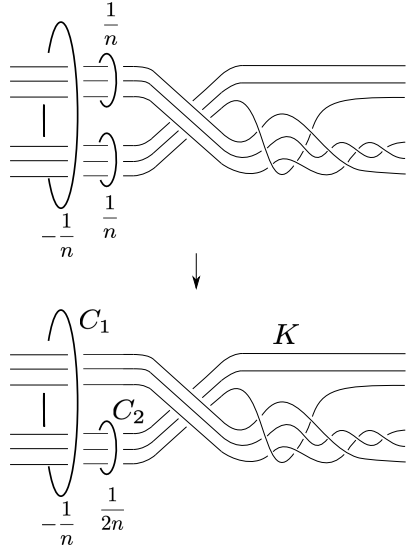}
\caption{A surgery diagram $K\cup C_1\cup C_2$.
Performing $-1/n$--surgery on $C_1$ and $1/2n$--surgery on $C_2$ changes $K$ to our $K_n$.}\label{fig:surgery}
\end{center}
\end{figure}

\begin{figure}[tb]
\begin{center}
\includegraphics[scale=0.4]{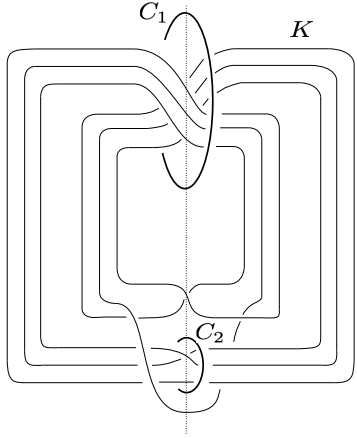}
\caption{A strongly invertible position of the link $K\cup C_1\cup C_2$ with an axis.}\label{fig:mont1}
\end{center}
\end{figure}

\begin{theorem}\label{thm:lspace}
For $n\ge 1$, $(18n+22)$--surgery on $K_n$ yields an L--space.
\end{theorem}

\begin{proof}
For the link in Fig.~\ref{fig:mont1},
take a quotient by the involution around the axis.
Note that the surgery coefficient on $K$ is $22$.
In the diagram of Fig.~\ref{fig:mont1}, the blackboard framing (or writhe) of $K$ is $19$.
Hence, $22$--surgery on $K$ corresponds to the tangle replacement by the $3$--tangle.

Rational tangle replacements corresponding to the surgeries on $K, C_1,C_2$ yield a link $\ell$, whose
double branched cover is the result of $(18n+22)$--surgery on $K_n$.
Figures \ref{fig:mont2}, \ref{fig:mont3}, \ref{fig:mont4} and \ref{fig:mont5} show the deformation of $\ell$.

\begin{figure}[tb]
\begin{center}
\includegraphics[scale=0.4]{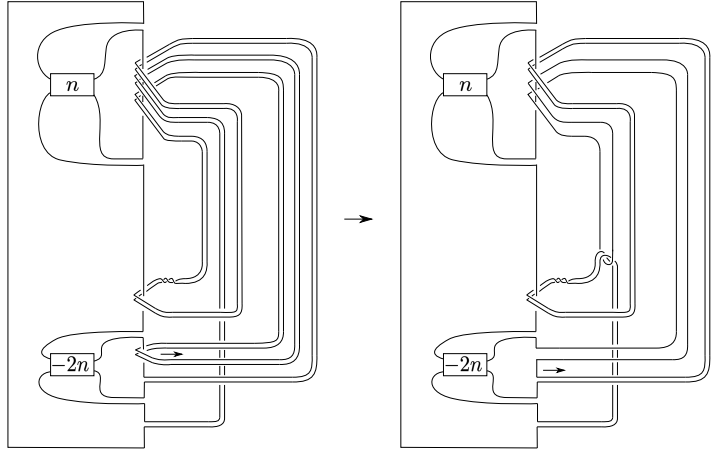}
\caption{A deformation of $\ell$, where a rectangle box means horizontal half-twists.
The integer indicates the number of half-twists, which is right-handed if it is positive, or left-handed otherwise.
}\label{fig:mont2}
\end{center}
\end{figure}

\begin{figure}[tb]
\begin{center}
\includegraphics[scale=0.4]{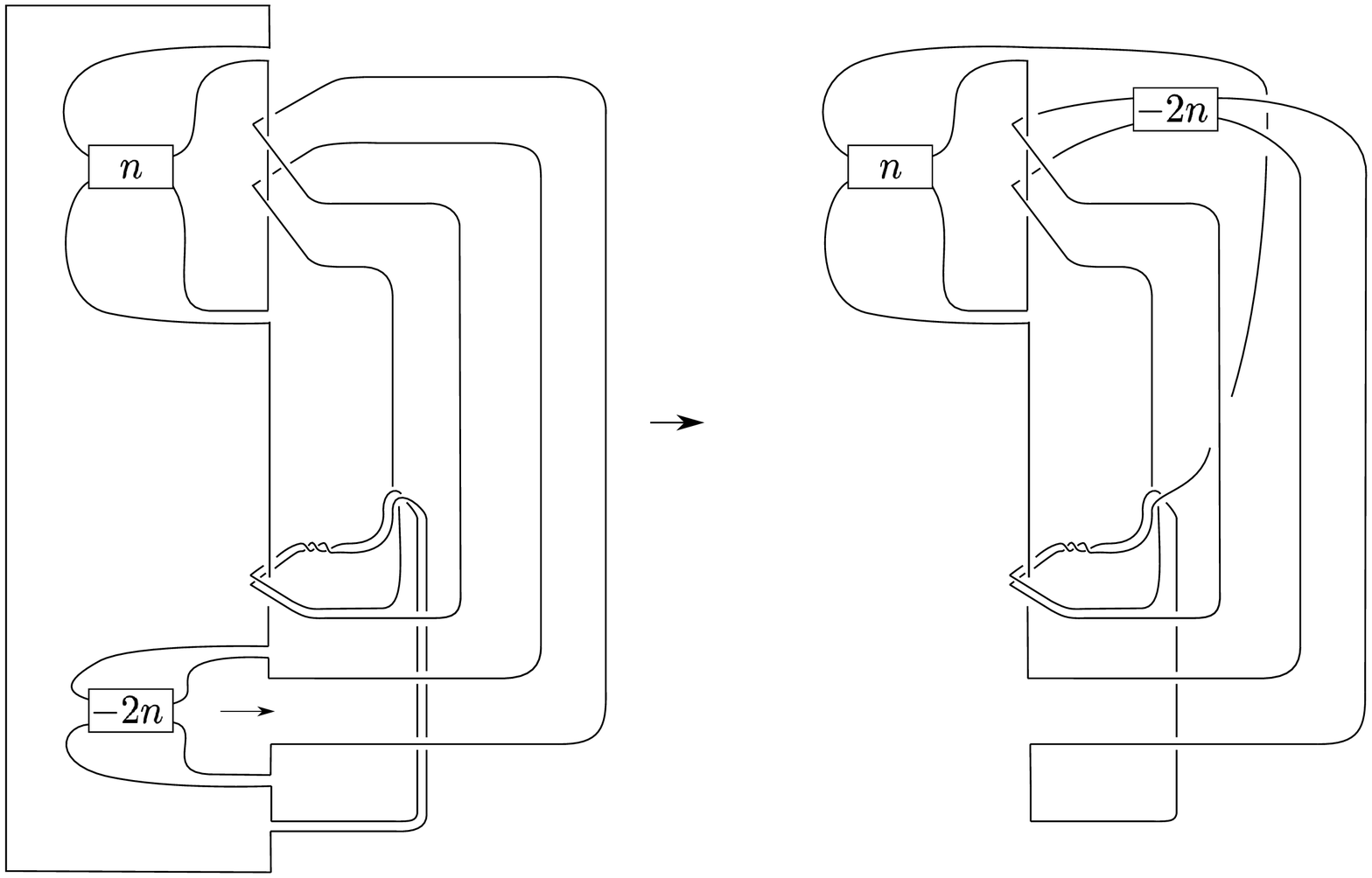}
\caption{Continued from Figure \ref{fig:mont2}.}\label{fig:mont3}
\end{center}
\end{figure}

\begin{figure}[tb]
\begin{center}
\includegraphics[scale=0.4]{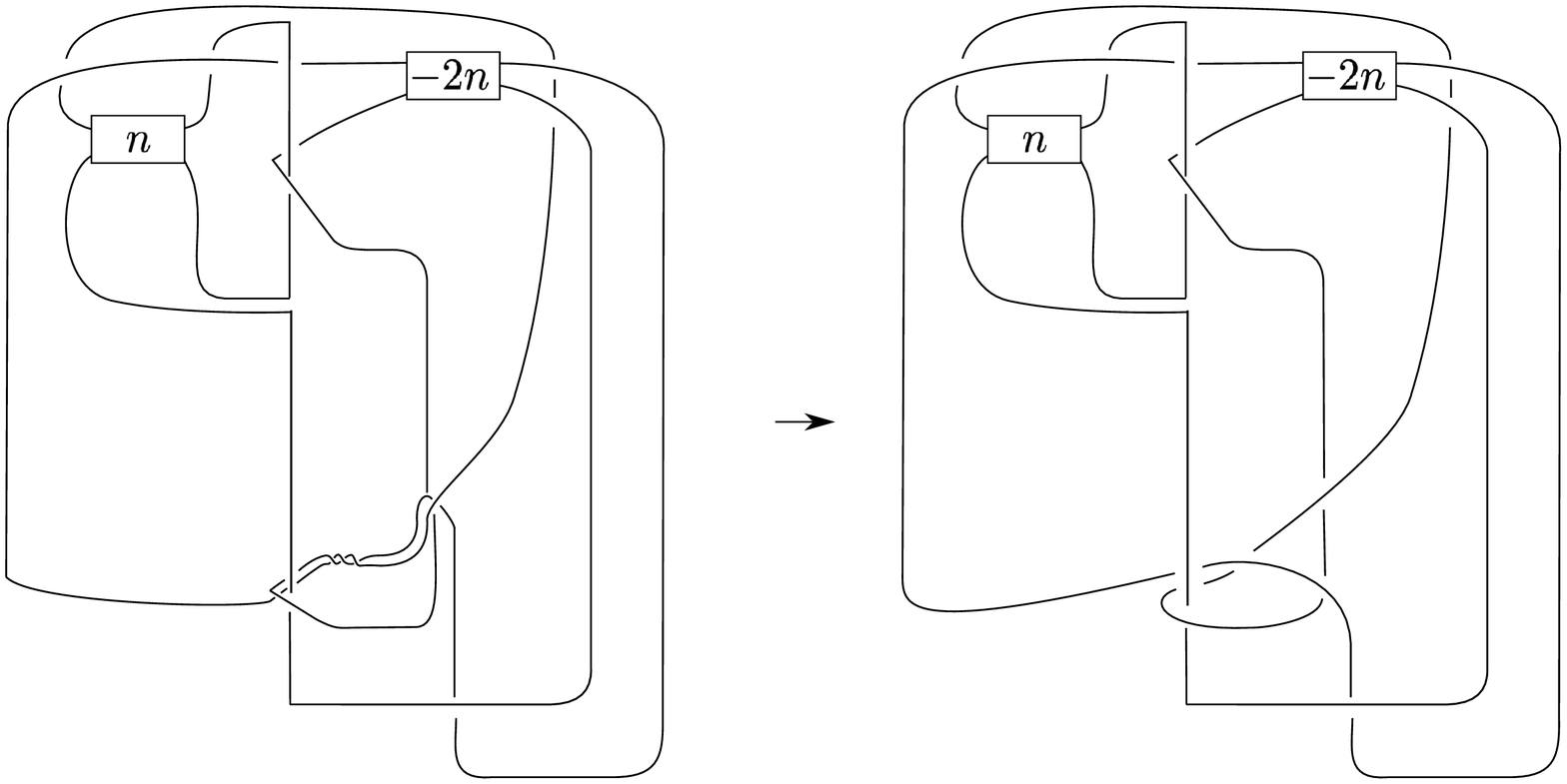}
\caption{Continued from Figure \ref{fig:mont3}.}\label{fig:mont4}
\end{center}
\end{figure}

\begin{figure}[tb]
\begin{center}
\includegraphics[scale=0.4]{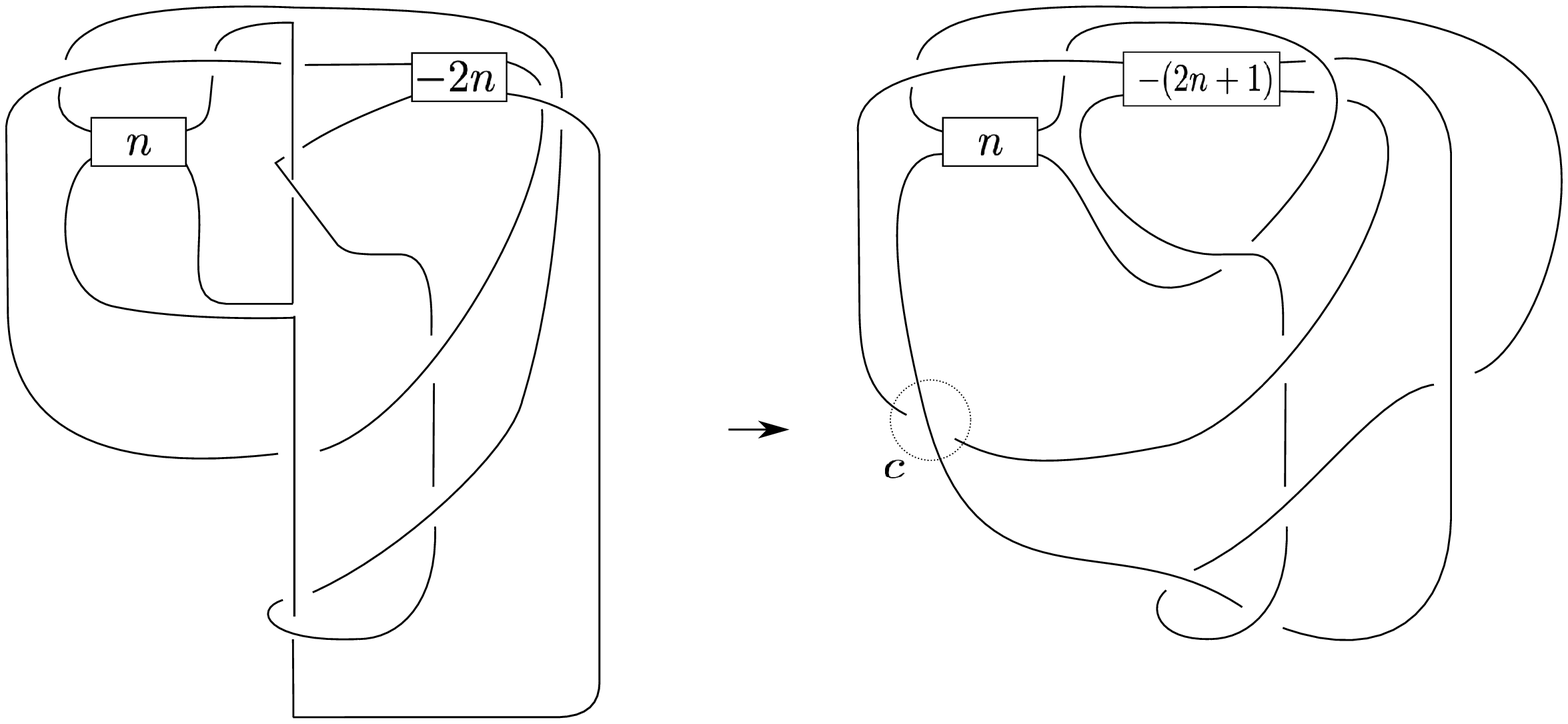}
\caption{Continued from Figure \ref{fig:mont4}.}\label{fig:mont5}
\end{center}
\end{figure}


In the next lemma, we prove that the double branched cover of the link $\ell$ is an L--space, so the proof is complete.
\end{proof}

\begin{lemma}
The double branched cover of the link $\ell$ is an L--space.
\end{lemma}

\begin{proof}
First, $\det \ell=18n+22$.
This is easily calculated from the Goeritz matrix for the checkerboard coloring of the diagram of $\ell$ in Fig.~\ref{fig:mont5}.
Let $c$ be the crossing as indicated in Fig.~\ref{fig:mont5}.
Then we have $\ell_0$ and $\ell_\infty$ by smoothing $c$ as shown in Fig.~\ref{fig:resolution}, and it is also a direct calculation to see
$\det \ell_0=4n+5$ and $\det \ell_\infty=14n+17$ from Figs.~\ref{fig:ell0} and \ref{fig:ellinf}.
Hence $\det \ell=\det \ell_0+\det \ell_\infty$ holds.

By \cite[Proposition 2.1]{OS2}, the triple of thee double branched covers of $\ell$, $\ell_0$ and $\ell_\infty$ forms a triad.
Thus  \cite[Proposition 2.1]{OS} claims that
if the double branched covers of $\ell_0$ and $\ell_\infty$ are L--spaces, then so is the double branched cover of $\ell$.

\begin{figure}[tb]
\begin{center}
\includegraphics[scale=0.6]{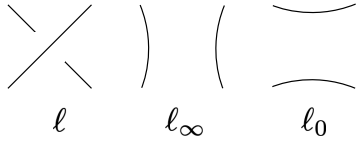}
\caption{Two resolutions.}\label{fig:resolution}
\end{center}
\end{figure}

\begin{figure}[tb]
\begin{center}
\includegraphics[scale=0.4]{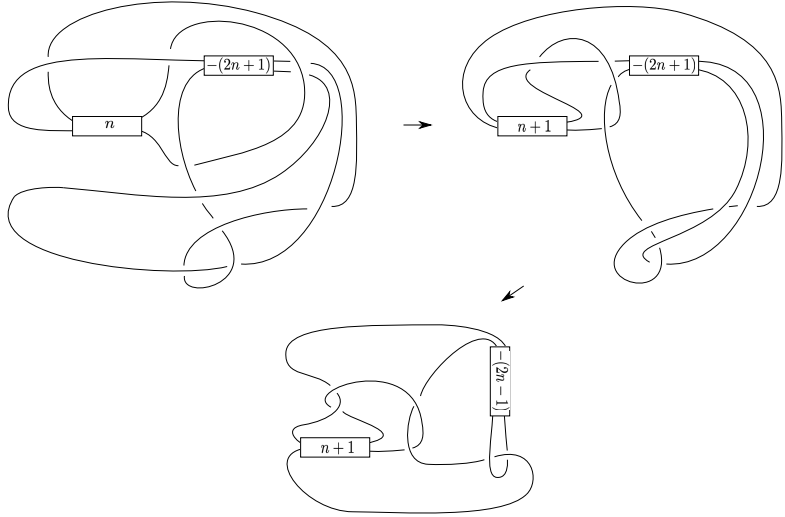}
\caption{The resolution $\ell_0$ is a Montesinos knot.}\label{fig:ell0}
\end{center}
\end{figure}

\begin{figure}[tb]
\begin{center}
\includegraphics[scale=0.4]{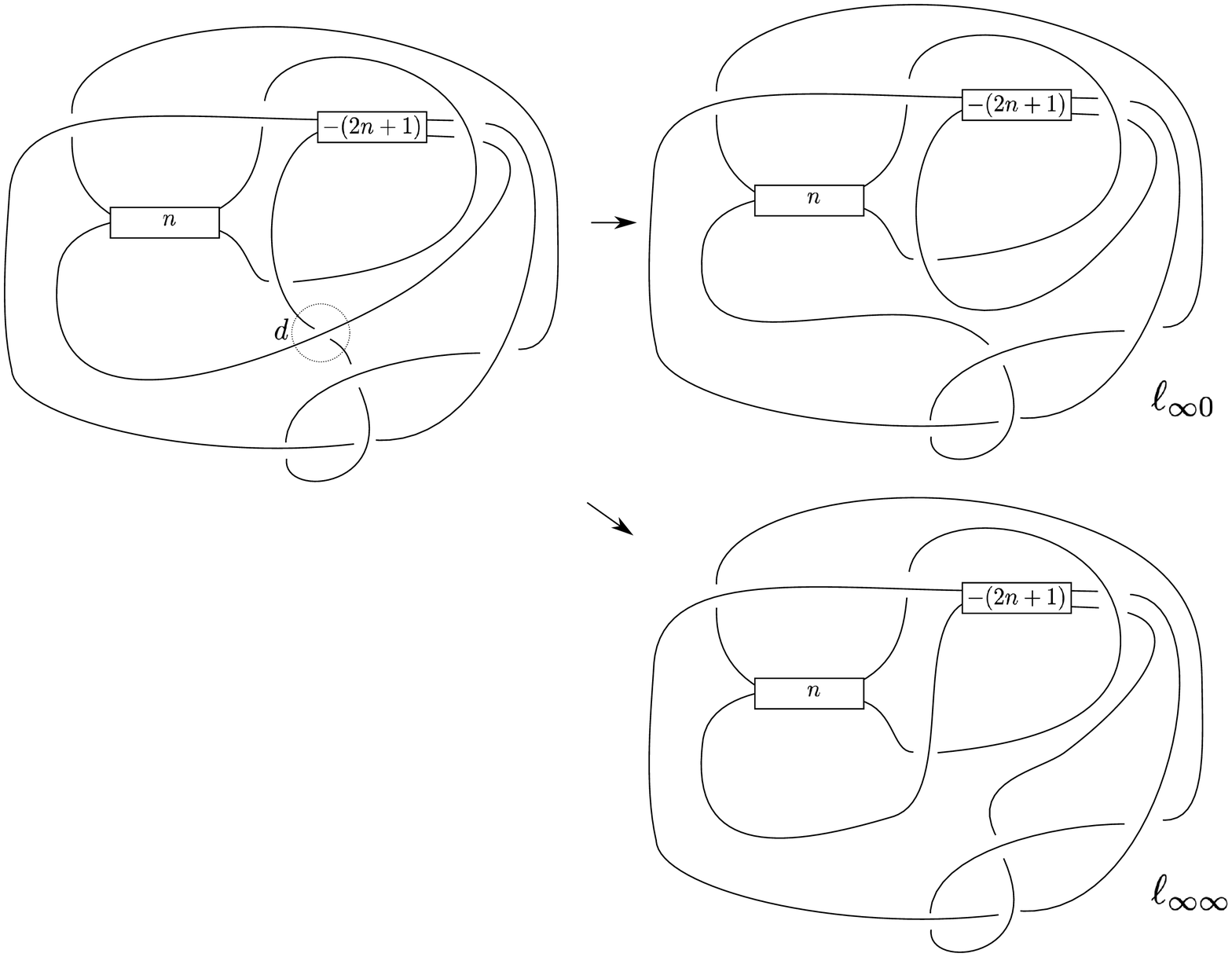}
\caption{The resolution $\ell_\infty$ and its further resolutions $\ell_{\infty 0}$ and $\ell_{\infty\infty}$ at the crossing $d$.}\label{fig:ellinf}
\end{center}
\end{figure}

We can see that $\ell_0$ is the Montesinos knot $M(-\frac{n+1}{2n+1}, \frac{1}{2}, \frac{2}{4n-1})$
as shown in Fig.~\ref{fig:ell0}.
When $n=1$, $\ell_0$ is the knot \texttt{8\_20}, which is non-alternating, but quasi-alternating.
Then the double branched cover is an L--space by \cite{OS2}.
If $n>1$, then $\ell_0$ is not quasi-alternating by \cite{I}.
Nevertheless, we can show the following.

\begin{claim}\label{cl:ell0}
The double branched cover of $\ell_0$ is an L--space.
\end{claim}

\begin{proof}[Proof of Claim \ref{cl:ell0}]
The double branched cover of $\ell_0$ is the Seifert fibered space $M(0;-\frac{n+1}{2n+1}, \frac{1}{2}, \frac{2}{4n-1})
=M(-1; \frac{n}{2n+1}, \frac{1}{2}, \frac{2}{4n-1})$.
(We use the convention of \cite{LS}, which is the same as \cite{BK}.)

Theorem 1.1 of \cite{LS} combined with  \cite{LM} claims that such a Seifert fibered space $M(-1;r_1,r_2,r_3)\ (1\ge r_1\ge r_2\ge r_3>0)$
is an L--space if and only if  there are no relatively prime integers $m>a>0$ such that 
$mr_1<a<m(1-r_2)$ and $mr_3<1$.

First, assume $n=1$.
Then $r_1=2/3$, $r_2=1/2$ and $r_3=1/3$, so $1-r_2=1/2$.
Since $r_1>1-r_2$, we have no solution $m$ satisfying $mr_1<m(1-r_2)$.

Suppose $n\ge 2$.
Then $r_1=1/2$, $r_2=n/(2n+1)$ and $r_3=2/(4n-1)$, so $1-r_2=(n+1)/(2n+1)$.

We assume that there are coprime integers $m$ and $a$ such that $m/2<a<m(n+1)/(2n+1)$ and $2m/(4n-1)<1$.
Then the first gives
\[
0<2a-m<\frac{m}{2n+1},
\]
and the second gives $m<2n-1/2$.
Combining these yields
\[
0<2a-m<\frac{4n-1}{4n+2}<1.
\]
Since $a$ and $m$ are integers, this is a  contradiction.
\end{proof}

For the other resolution $\ell_\infty$, we further perform two smoothings at the crossing $d$ as shown in Fig.~\ref{fig:ellinf}.
Then we have a link $\ell_{\infty 0}$ and a knot $\ell_{\infty\infty}$ as shown there.
In particular, a direct calculation on Fig.~\ref{fig:ellinf} (or, Figs.~\ref{fig:ellinf0} and \ref{fig:ellinfinf}) shows that $\det\ell_{\infty 0}=4n+14$ and $\det \ell_{\infty\infty}=10n+3$.
Hence $\det\ell_\infty=\det\ell_{\infty 0}+\det\ell_{\infty \infty}$ holds.

\begin{figure}[tb]
\begin{center}
\includegraphics[scale=0.4]{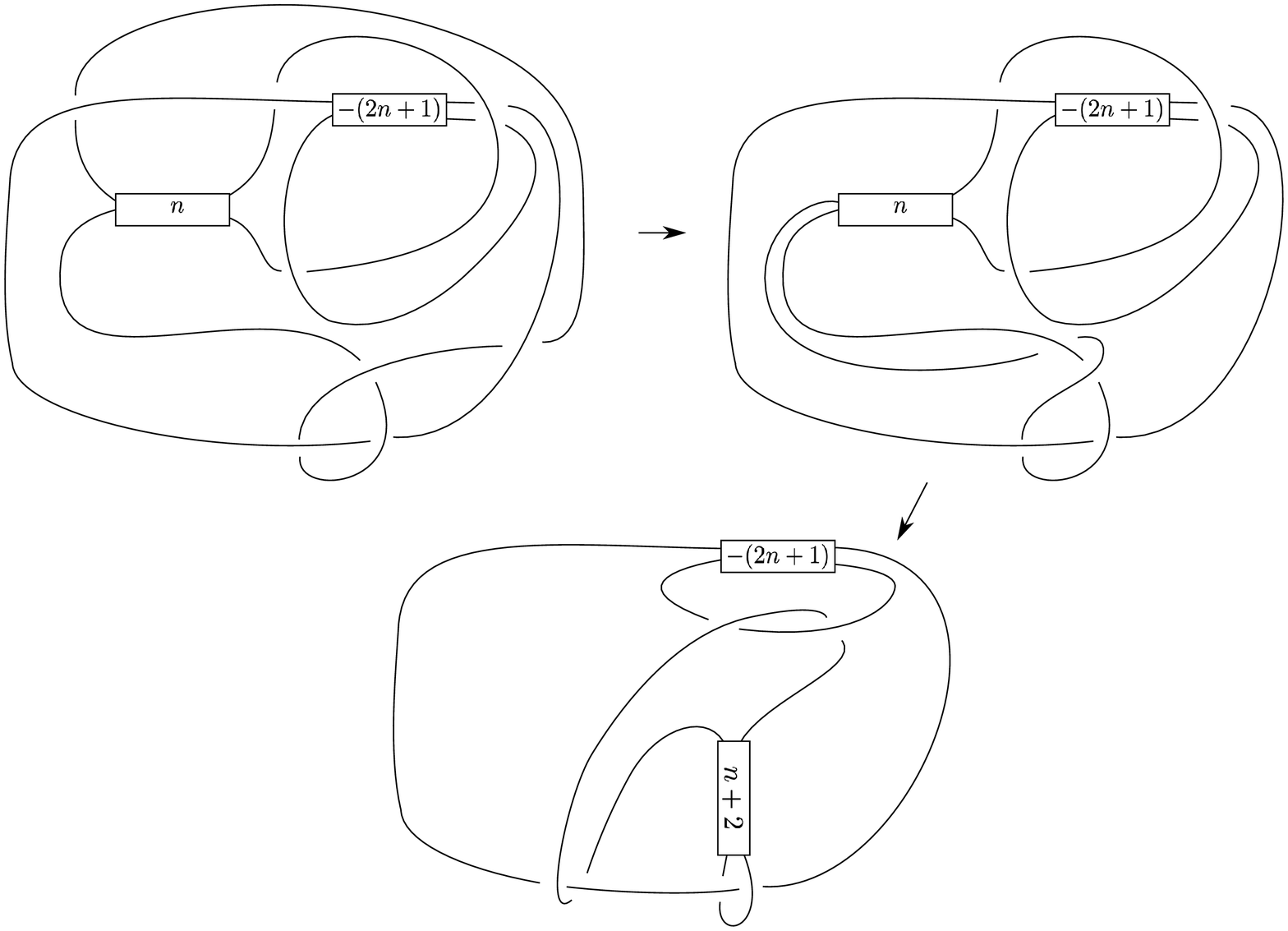}
\caption{The resolution $\ell_{\infty 0}$ is a Montesinos link.}\label{fig:ellinf0}
\end{center}
\end{figure}

\begin{claim}\label{cl:ellinf}
The double branched covers of $\ell_{\infty 0}$ and $\ell_{\infty \infty}$ are L--spaces.
\end{claim}

\begin{proof}[Proof of Claim \ref{cl:ellinf}]
The link $\ell_{\infty 0}$ is the Montesinos link $M(\frac{1}{2n+1},-\frac{1}{2},\frac{2n+3}{4n+8})$ as shown in Fig.~\ref{fig:ellinf0}.
Although this is not quasi-alternating by \cite{I}, we can show that the double branched cover is an L--space as before.
The double branched cover is the Seifert fibered space $M(0;\frac{1}{2n+1},-\frac{1}{2},\frac{2n+3}{4n+8})=
M(-1; \frac{1}{2n+1},\frac{1}{2},\frac{2n+3}{4n+8})$.
As in the proof of Claim \ref{cl:ell0}, set $r_1=1/2$, $r_2=(2n+3)/(4n+8)$ and $r_3=1/(2n+1)$.

Suppose that there are coprime integers $m$ and $a$ such that
$mr_1<a<m(1-r_2)$ and $mr_3<1$.
Then 
\[
0<2a-m<\frac{m}{2n+4}<\frac{2n+1}{2n+4}<1,
\]
a contradiction.


The link $\ell_{\infty \infty}$ is the Montesinos knot $M(-\frac{1}{2},\frac{n}{2n+1},\frac{5}{10n+7})$ as shown in Fig.~\ref{fig:ellinfinf}, which
is not quasi-alternating by \cite{I} again.
The double branched cover is the Seifert fibered space $M(0;-\frac{1}{2},\frac{n}{2n+1},\frac{5}{10n+7})=
M(-1; \frac{1}{2}, \frac{n}{2n+1},\frac{5}{10n+7})$.

\begin{figure}[tb]
\begin{center}
\includegraphics[scale=0.4]{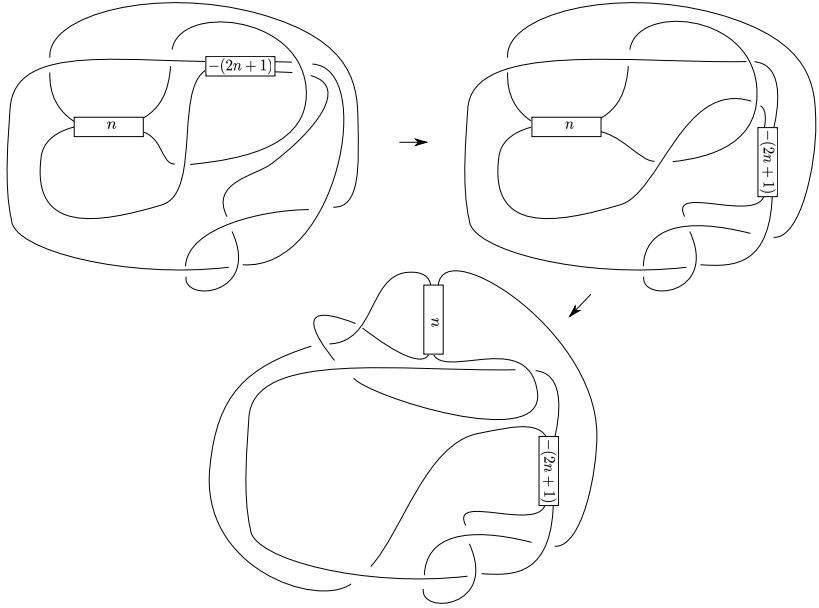}
\caption{The resolution $\ell_{\infty \infty}$ is a Montesinos knot.}\label{fig:ellinfinf}
\end{center}
\end{figure}

Set $r_1=1/2$, $r_2=n/(2n+1)$ and $r_3=5/(10n+7)$.
Suppose that there are coprime integers $m$ and $a$ as above.
Then 
\[
0<2a-m<\frac{m}{2n+1}<\frac{10n+7}{10n+5}.
\]
Hence $2a-m=1$.

Then $a<m(1-r_2)$ implies $2n+1<m$.
Combining with $mr_3<1$ gives 
\[
10n+5<5m<10n+7,
\]
which is impossible.
\end{proof}

By \cite[Proposition 2.1]{OS} and \cite[Proposition 2.1]{OS2}, the double branched cover of $\ell_\infty$ is an L--space, so is that of $\ell$.
\end{proof}

We remark that 
computer experiments suggest that the knot $K_n$ does not admit a nontrivial exceptional surgery.
This situation brings us a difficulty to find candidates of slopes for L--space surgeries.
Also, since $K_n$ has genus $9n+7$,
if $r$--surgery on $K_n$ yields an L--space, then $r\ge 2(9n+7)-1=18n+13$ by \cite{OS3}.
In fact, it is known that any $r\ (\ge 18n+13)$ yields an L--space.
We selected the slope $18n+22$ for our proof, but there might be a better slope for a proof.

\section{Alexander polynomials and formal semigroups}\label{sec:alex}

In this section, we calculate the Alexander polynomial $\Delta_{K_n}(t)$ of $K_n$, and
its formal semigroup.
For the former, we mimic the argument in \cite{BK,BM}.

\begin{theorem}\label{thm:alexander}
The Alexander polynomial of $K_n$ is given as
\[
\Delta_{K_n}(t)=t^{6n+4}+\sum_{i=0}^n(A_1+A_2+A_3+A_4+A_5),
\]
where
\begin{align*}
A_1&=t^{6(n-i)}-t^{6(n-i)+1},\\
A_2&=t^{6(n+i)+6}-t^{6(n+i)+5},\\
A_3&=t^{6(n+i)+8}-t^{6(n+i)+7},\\
A_4&=t^{6(n+i)+10}-t^{6(n+i)+9},\\
A_5&=t^{6(2n+i)+14}-t^{6(2n+i)+13}.
\end{align*}
\end{theorem}

\begin{proof}

Let $L=K\cup C_1\cup C_2$ be the oriented link as shown in Fig.~\ref{fig:surgery2}.
We remark that this is modified from the link in Fig.~\ref{fig:surgery} to reduce the number of crossing
by changing the surgery coefficients.
(This process is not critical, because we only need to calculate the multivariable Alexander polynomial.)

\begin{figure}[tb]
\begin{center}
\includegraphics[scale=0.4]{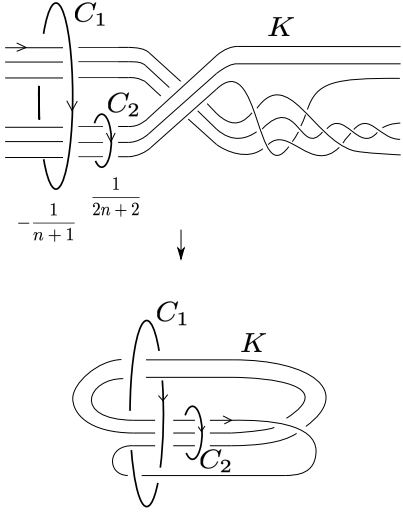}
\caption{A modified surgery diagram of $L=K\cup C_1\cup C_2$.
This is simpler than the previous link in Figure \ref{fig:surgery}.}\label{fig:surgery2}
\end{center}
\end{figure}

 It has the multivariable Alexander polynomial
 \[
 \begin{split}
 \Delta_{L}(x,y,z)&=
 (x^3-1)(x^5y^4z^2+x^3y^5z^2-x^3y^4z^2+x^2y^4z^2
+x^4y^2z+x^3y^3z\\
&\quad -x^3y^2z-x^2y^3z+x^2y^2z+xy^3z+x^3y-x^2y+x^2+y).
 \end{split}
\]
(We used \cite{Ko} for the calculation.)

Performing $-1/(n+1)$--surgery on $C_1$ and $1/(2n+2)$--surgery on $C_2$ changes
the link $K\cup C_1\cup C_2$ to $K_n\cup C_1^n\cup C_2^n$.
Clearly, these links have homeomorphic exteriors.
Hence the induced isomorphism of the homeomorphism on their homology groups relates the Alexander polynomials
of two links.  (See \cite{F, M}.)
Let $\mu_K$, $\mu_{C_1}$ and $\mu_{C_2}$ be the homology classes of meridians of $K$, $C_1$ and $C_2$, respectively.
We assume that each (oriented) meridian has linking number one with the corresponding knot.
Moreover, let $\lambda_K$, $\lambda_{C_1}$ and $\lambda_{C_2}$ be the homology classes of their oriented longitudes.

Similarly, we have homology classes of meridians, $\mu_{K_n}$, $\mu_{C_1^n}$ and $\mu_{C_2^n}$ of $K_n$, $C_1^n$ and $C_2^n$. 
Then we have
\[
\mu_{K_n}=\mu_{K},\quad  \mu_{C_1^n}=\mu_{C_1}-(n+1)\lambda_{C_1},\quad  \mu_{C_2^n}=\mu_{C_2}+(2n+2)\lambda_{C_2}.
\]
Since $\lambda_{C_1}=6\mu_K$ and $\lambda_{C_2}=3\mu_K$,
\[
\mu_{C_1^n}=\mu_{C_1}-6(n+1)\mu_K,\quad \mu_{C_2^n}=\mu_{C_2}+6(n+1)\mu_K.
\]
Thus
\[
\mu_{K_n}=\mu_K, \quad \mu_{C_1}=\mu_{C_1^n}+6(n+1)\mu_K, \quad \mu_{C_2}=\mu_{C_2^n}-6(n+1)\mu_K.
\]
Hence we have the relation between the Alexander polynomials as
\begin{equation}\label{eq:relation-alexander}
\Delta_{K_n\cup C_1^n\cup C_2^n}(x,y,z)=\Delta_{L}(x,yx^{6(n+1)},zx^{-6(n+1)} ).
\end{equation}

On the other hand, 
since $\mathrm{lk}(K_n,C_2^n)=\mathrm{lk}(K,C_2)=3$,
the Torres condition \cite{T} gives
\begin{align*}
\Delta_{K_n\cup C_1^n\cup C_2^n}(x,y,1)&=(x^3y^0-1)\Delta_{K_n\cup C_1^n}(x,y)\\
&=(x^3-1)\Delta_{K_n\cup C_1^n}(x,y).
\end{align*}
Similarly, since $\mathrm{lk}(K_n,C_1^n)=\mathrm{lk}(K,C_1)=6$,
\[
\Delta_{K_n\cup C_1^n}(x,1)=\frac{x^6-1}{x-1}\Delta_{K_n}(x).
\]
Thus,
\[
\Delta_{K_n}(x)=\frac{x-1}{x^6-1}\Delta_{K_n\cup C_1^n}(x,1)=\frac{x-1}{(x^6-1)(x^3-1)}\Delta_{K_n\cup C_1^n\cup C_2^n}(x,1,1).
\]
Then (\ref{eq:relation-alexander}) gives
\begin{align*}
\Delta_{K_n}(t)&=\frac{t-1}{(t^6-1)(t^3-1)}\Delta_L(t,t^{6(n+1)},t^{-6(n+1)})\\
&=\frac{t^2(t-1)}{t^6-1}( t^{18n+19}+t^{12n+15}+t^{12n+11}+t^{6n+8}+t^{6n+4}+1 )\\
&\stackrel{.}{=}\frac{ t^{18n+19}+t^{12n+15}+t^{12n+11}+t^{6n+8}+t^{6n+4}+1}{t^5+t^4+t^3+t^2+t+1}.
\end{align*}

(Recall that $\stackrel{.}{=}$ means equivalence up to units.)

Next,  we calculate 
 \begin{equation}\label{eq:alex}
(t^5+t^4+t^3+t^2+t+1)\bigl(t^{6n+4}+\sum_{i=0}^n(A_1+A_2+A_3+A_4+A_5)\bigr).
\end{equation}
First, 
\[
(t^5+t^4+t^3+t^2+t+1)t^{6n+4}=t^{6n+9}+t^{6n+8}+t^{6n+7}+t^{6n+6}+t^{6n+5}+t^{6n+4}.
\]
Next, 
\begin{align*}
(t^5+t^4+t^3+t^2+t+1)\sum_{i=0}^n A_1&=(t^5+t^4+t^3+t^2+t+1)\sum_{i=0}^n(t^{6(n-i)}-t^{6(n-i)+1})\\
&=(t^5+t^4+t^3+t^2+t+1)\sum_{i=0}^n t^{6(n-i)}(1-t)\\
&=(1-t^6)\sum_{i=0}^n t^{6(n-i)}\\
&=\sum_{i=0}^n t^{6(n-i)}-\sum_{i=0}^n t^{6(n-i)+6}\\
&=1-t^{6n+6}.
\end{align*}

Similarly,
\begin{align*}
(t^5+t^4+t^3+t^2+t+1)\sum_{i=0}^n A_2&=(t^5+t^4+t^3+t^2+t+1)\sum_{i=0}^n(t^{6(n+i)+6}-t^{6(n+i)+5})\\
&=(t^6-1)\sum_{i=0}^n t^{6(n+i)+5}\\
&=\sum_{i=0}^n t^{6(n+i)+11}-\sum_{i=0}^n t^{6(n+i)+5}\\
&=t^{12n+11}-t^{6n+5},
\end{align*}
\begin{align*}
(t^5+t^4+t^3+t^2+t+1)\sum_{i=0}^n A_3&=(t^5+t^4+t^3+t^2+t+1)\sum_{i=0}^n(t^{6(n+i)+8}-t^{6(n+i)+7})\\
&=(t^6-1)\sum_{i=0}^n t^{6(n+i)+7}\\
&=\sum_{i=0}^n t^{6(n+i)+13}-\sum_{i=0}^n t^{6(n+i)+7}\\
&=t^{12n+13}-t^{6n+7},
\end{align*}
\begin{align*}
(t^5+t^4+t^3+t^2+t+1)\sum_{i=0}^n A_4&=(t^5+t^4+t^3+t^2+t+1)\sum_{i=0}^n(t^{6(n+i)+10}-t^{6(n+i)+9})\\
&=(t^6-1)\sum_{i=0}^n t^{6(n+i)+9}\\
&=\sum_{i=0}^n t^{6(n+i)+15}-\sum_{i=0}^n t^{6(n+i)+9}\\
&=t^{12n+15}-t^{6n+9},
\end{align*}
\begin{align*}
(t^5+t^4+t^3+t^2+t+1)\sum_{i=0}^n A_5&=(t^5+t^4+t^3+t^2+t+1)\sum_{i=0}^n(t^{6(2n+i)+14}-t^{6(2n+i)+13})\\
&=(t^6-1)\sum_{i=0}^n t^{6(2n+i)+13}\\
&=\sum_{i=0}^n t^{6(2n+i)+19}-\sum_{i=0}^n t^{6(2n+i)+13}\\
&=t^{18n+19}-t^{12n+13}.
\end{align*}

These show that (\ref{eq:alex}) is equal to 
$t^{18n+19}+t^{12n+15}+t^{12n+11}+t^{6n+8}+t^{6n+4}+1$. 
Hence 
\[
\Delta_{K_n}(t)=t^{6n+4}+\sum_{i=0}^n(A_1+A_2+A_3+A_4+A_5).
\]
\end{proof}

\begin{lemma}\label{lem:rank}
For $n\ge 1$, let $\mathcal{T}=\langle 6,6n+4,6n+8,12n+11,12n+15\rangle$.
The the semigroup $\mathcal{T}$ has rank $5$.
\end{lemma}

\begin{proof}
Let $G$ be a generating set of $\mathcal{T}$.
It suffices to show that $\{6,6n+4,6n+8,12n+11,12n+15\}\subset G$.

Since $6$ is the minimal nonzero element of $\mathcal{T}$, we need $6\in G$.
Except the multiples of $6$, $6n+4$ is the minimal element of $\mathcal{T}$ and $6n+8$ is the next, so $6n+4, 6n+8\in G$.

Assume $12n+11=6a+(6n+4)b+(6n+8)c$ for $a,b,c\ge 0$.
Then $1\equiv 0\pmod{2}$, so $12n+11\not\in \langle 6,6n+4,6n+8\rangle$.
Hence $12n+11\in G$.

Finally, assume $12n+15=6a+(6n+4)b+(6n+8)c+(12n+11)d$ for  $a,b,c,d\ge 0$.
Then $1\equiv d\pmod{2}$, so $d\ne 0$.  In fact, $d=1$.
We have $4=6a+(6n+4)b+(6n+8)c$.
Since $n\ge 1$, this is impossible.
Hence $12n+15\in G$.
\end{proof}

\begin{theorem}\label{thm:formalsemigroup}
The formal semigroup $\mathcal{S}$ of $K_n$ is a semigroup of rank $5$\textup{:}
\[
\mathcal{S}=\langle 6,6n+4,6n+8,12n+11,12n+15\rangle.
\]
\end{theorem}

\begin{proof}
By Theorem \ref{thm:alexander},
\[
\Delta_{K_n}(t)=t^{6n+4}+\sum_{i=0}^n(A_1+A_2+A_3+A_4+A_5).
\]
Hence, as a formal power series,
\begin{align*}
\frac{\Delta_{K_n}(t)}{1-t}&=\frac{t^{6n+4}}{1-t}+\sum_{i=0}^n\Bigl(  
\frac{A_1}{1-t}+\frac{A_2}{1-t}+\frac{A_3}{1-t}+\frac{A_4}{1-t}+\frac{A_5}{1-t}
\Bigr)\\
&=t^{6n+4}\sum_{j=0}^\infty t^j+\sum_{i=0}^n(t^{6(n-i)}-t^{6(n+i)+5}-t^{6(n+i)+7}-t^{6(n+i)+9}-t^{6(2n+i)+13}  )\\
&=\sum_{j=0}^\infty t^{6n+4+j}+\sum_{i=0}^n t^{6(n-i)} - t^{6n+5}\sum_{i=0}^{n}(t^{6i}+t^{6i+2}+t^{6i+4})-\sum_{i=0}^nt^{12n+13+6i}.
\end{align*}
Then
\[
\mathcal{S}=\mathbb{Z}_{\ge 6n+4}\cup B-C-D,
\]
where $B=\{0,6,12,\dots,6n\}$, $C=\{6n+5,6n+7,6n+9,\dots,12n+5,12n+7,12n+9\}$ and
$D=\{12n+13,12n+19,\dots,18n+13\}$.

Let $\mathcal{T}=\langle 6,6n+4,6n+8,12n+11,12n+15\rangle$.
We need to show that $\mathcal{S}=\mathcal{T}$.

First, if $m\ge 18n+14$, then $m\in \mathcal{S}$.   Thus $\mathbb{Z}_{\ge 18n+14}\subset \mathcal{S}$.
To show that $\mathbb{Z}_{\ge 18n+14}\subset \mathcal{T}$, it suffices to verify that
$18n+14,18n+15,\dots,18n+19 \in \mathcal{T}$, since $6\in \mathcal{T}$.
This follows from
\[
18n+14\equiv 6n+8,18n+15\equiv 12n+15,  18n+16\equiv 6n+4,
\]
\[
18n+17\equiv 12n+11, 18n+18\equiv 6 \pmod{6},\]
\[
18n+19=(6n+8)+(12n+11). 
\]

Next, the set $\mathbb{Z}_{<18n+14}$ of nonnegative integers less than $18n+14$ is
decomposed into the congruence classes of modulo 6, $Z_0,Z_1,Z_2,\dots,Z_5$, where
$Z_i=\{m\mid 0\le m<18n+14, m\equiv i \pmod{6}\}$.
Then $B\subset Z_0$, $D\subset Z_1$, and $C\subset Z_1\cup Z_3\cup Z_5$.

We examine each congruence class.
\begin{itemize}
\item
$Z_0\subset \mathcal{S}$ and $Z_0\subset \mathcal{T}$.  Thus $Z_0\cap \mathcal{S}=Z_0\cap \mathcal{T}=Z_0$.
\item
$Z_1\cap \mathcal{S}=\varnothing$.
\item
$Z_2\cap \mathcal{S}=\{6n+8,6n+14,\dots, 18n+8\}\subset \mathcal{T}$.
\item
$Z_3\cap\mathcal{S}=\{ 12n+15,12n+21,\dots, 18n+9 \}\subset \mathcal{T}$.
\item
$Z_4\cap \mathcal{S}=\{6n+4,6n+10,\dots,18n+10\}\subset \mathcal{T}$.
\item
$Z_5\cap \mathcal{S}=\{12n+11,12n+17,\dots, 18n+11  \}\subset \mathcal{T}$.
\end{itemize}
Hence $\mathcal{S}\cap \mathbb{Z}_{<18n+14}\subset \mathcal{T}$.

Conversely, let $m\in Z_1$.  
If $m\in \mathcal{T}$, then 
we need to use $12n+11$ or $12n+15$ to yield $m$.
Since $m\le 18n+13$,  either
\begin{itemize}
\item[(1)] $m=(12n+11)+r$ and $r\le 6n+2$, $r\equiv 2\pmod{6}$, or
\item[(2)]  $m=(12n+15)+r$ and $r\le 6n-2$, $r\equiv 4 \pmod{6}$.
\end{itemize}
However, there is no $r\in \mathcal{T}$ satisfying these.  Hence $Z_1\cap \mathcal{T}=\varnothing$.

Let $m\in Z_2\cap \mathcal{T}$.
If $m<6n+8$, then $m\le 6n+2$. So, there is no such element in $\mathcal{T}$.
Hence $Z_2\cap \mathcal{S}=Z_2\cap \mathcal{T}$.

Similarly, let $m\in Z_4\cap \mathcal{T}$.
If $m<6n+4$, then $m\le 6n-2$.
But there is no such element in $\mathcal{T}$.
Hence $Z_4\cap \mathcal{S}=Z_4\cap \mathcal{T}$.

Let $m\in Z_3\cap \mathcal{T}$.
Again, we need to use $12n+11$ or $12n+15$ to yield $m$.
Then, we have either
\begin{itemize}
\item[(3)]
$m=(12n+11)+r$ and $r\le 6n-2$ and $r\equiv 4 \pmod{6}$, or
\item[(4)]
$m=(12n+15)+r$ and $r\le 6n-6$ and $r\equiv 0\pmod{6}$.
\end{itemize}
For (3), there is no $r\in \mathcal{T}$.
For (4),  $r\in \{0,6,12,\dots, 6n-6\}$, so $m\in Z_3\cap \mathcal{S}$.
Hence $Z_3\cap \mathcal{S}=Z_3\cap \mathcal{T}$.

Finally, let $m\in Z_5\cap \mathcal{T}$.
Again, we have either
\begin{itemize}
\item[(5)]
$m=(12n+11)+r$ and $r\le 6n$ and $r\equiv 0 \pmod{6}$, or
\item[(6)]
$m=(12n+15)+r$ and $r\le 6n-4$ and $r\equiv 2\pmod{6}$.
\end{itemize}
For (6), there is no $r\in \mathcal{T}$.
For (5), $r\in \{0,6,12,\dots,6n\}$, so $m\in Z_5\cap \mathcal{S}$.
Hence $Z_5\cap \mathcal{S}=Z_5\cap \mathcal{T}$.
\end{proof}

\section{Hyperbolicity}\label{sec:hyp}

In this section, we prove that our knot $K_n$ is a hyperbolic knot for any $n\ge 1$
by using the fact that $K_n$ has tunnel number one.

\begin{lemma}\label{lem:tunnel}
For $n\ge 1$, $K_n$ has tunnel number one, hence $K_n$  is prime.
\end{lemma}

\begin{proof}
Figure \ref{fig:tunnel} shows an unknotting tunnel $\gamma$ for $K_n$.
The series of isotopy as illustrated in Fig.~\ref{fig:tunnel} indicates that
the outside of a regular neighborhood of $K_n\cup \gamma$ is a genus two handlebody.
\end{proof}

\begin{figure}[tb]
\begin{center}
\includegraphics[scale=0.48]{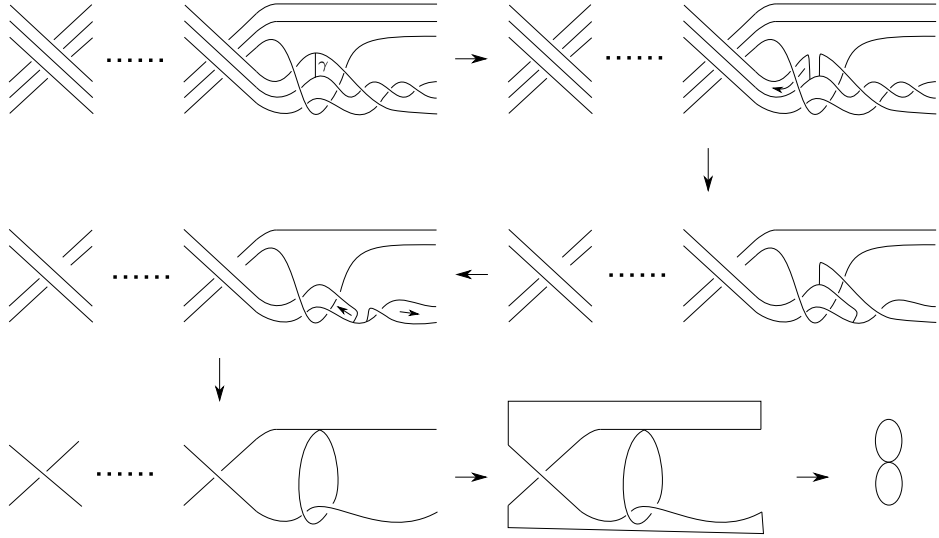}
\caption{The unknotting tunnel $\gamma$ and the series of isotopy of $N(K_n\cup \gamma)$.}\label{fig:tunnel}
\end{center}
\end{figure}

\begin{theorem}\label{thm:hyp}
For $n\ge 1$, $K_n$ is a hyperbolic knot.
\end{theorem}

\begin{proof}
By Theorem \ref{thm:formalsemigroup}, the formal semigroup of $K_n$ is a semigroup of rank $5$.
Since the formal semigroup of a torus knot is a semigroup of rank two (see \cite{BL}), $K_n$ is not a torus knot.

Assume that $K_n$ is a satellite knot for a contradiction.
By Lemma \ref{lem:tunnel}, $K_n$ has tunnel number one.
Then Morimoto and Sakuma's classification \cite{MS} tells us that $K_n$ has a torus knot $T(p,q)$ as its companion.
Since $K_n$ has bridge number at most $6$, the companion has bridge number at most three \cite{Sc}.
More precisely, either the companion is $3$--bridge and the wrapping number of the pattern is two, or
the companion is $2$--bridge and the wrapping number is two or three.

By Theorem \ref{thm:lspace}, $K_n$ is an L--space knot.
Then \cite{HRW} implies that the pattern knot is also an L--space knot.
Furthermore, \cite[Theorem 1.17]{BM} claims that the pattern is braided in the pattern solid torus.
In particular, the wrapping number coincides with the winding number there.

We divide the argument into two cases.

\medskip
\textbf{Case 1.}
Suppose that $K_n$ has a companion $T(3,q)\ (|q|>3)$ and a braided pattern knot $P$.
Then the wrapping number and winding number of $P$ are equal to two.
This means that $K_n$ is a $2$--cable of $T(3,q)$.
However, the cabling formula of \cite{W} shows that its formal semigroup has rank three.
Hence this case is impossible.

\medskip 
\textbf{Case 2.}
Suppose that $K_n$ has a companion $T=T(2,q)\ (|q|\ge 3)$ and a braided pattern knot $P$.
We may assume that $q>0$ by taking the mirror image of $K_n$, if necessary.
If the wrapping number is two, then $K_n$ is a $2$--cable, which is a contradiction again.
Hence the wrapping number and the winding winding number are equal to three.
For the Alexander polynomials, we have
\[
\Delta_{K_n}(t)=\Delta_T(t^3)\Delta_P(t)
\]
(see \cite{BZ}).
Here, $\Delta_T(t^3)=1-t^3+t^6-\dots+t^{3(q-1)}$.
Also, this implies $g(K_n)=3g(T)+g(P)$  (see \cite[Lemma 2.6]{HLV}).


By \cite[Theorem 3.1]{LV}, the only closed $3$--braids which are L--space knots are
torus knots and twisted torus knots $T(3,t; 2,s)$ with $ts>0$.
Here, $T(3,t; 2,s)$ is obtained from $T(3,t)$ by adding $s$ full twists on two adjacent strings.
Again, a $3$--cable is excluded.

We need to recall the construction of \cite{MS} of tunnel number one satellite knots.
Let $k_1\cup k_2$ be a $2$--bridge link in $S^3$.
We remark that each component $k_i$ is unknotted.
The exterior of $k_2$ is a solid torus $J$ containing $k_1$ in its interior.
Here, the longitude of $J$ is the meridian of $k_2$.
For the companion $T$, consider the homeomorphism $f$ from $J$ to the regular neighborhood $N(T)$ of $T$, which sends the longitude of $J$ to
the $(1,2q)$--curve on $\partial N(T)$.
This $(1,2q)$--curve corresponds to a regular fiber of the Seifert fibration in the exterior of $T$.
Then the image $f(k_1)$ gives our $K_n$.

Since the pattern knot $P$ is defined so as to preserve the preferred longitudes of $J$ and $N(T)$,
$P$ is obtained from $k_1$ in $J$ by adding $2q$--full twists.
Conversely, if we add $(-2q)$-full twists on $P$, then the result is unknotted.

By the classification of twisted torus knots which are unknotted in \cite{L}, 
$T(3,2; 2,-1)$, $T(3,2; 2, -2)$, $T(3,1;2,-1)$ and their mirror images
$T(3,-2;2;1)$, $T(3,-2;2,2)$, $T(3,-1;2,1)$
 give all $3$--strand twisted torus knots, which are unknotted.
Thus Table 1 is the list of possible pattern knot $P$ with genus.
(Each knot has a positive braid presentation, so its genus is calculated as in Section 2.)
Since $P$ is an L--space knot and not a $3$--cable, (1), (2) and (3) are excluded by \cite{LV}.
(In fact, (1) gives a $3$--cable.)

\begin{table}[tbp]
\begin{tabular}{c|c|c}
 & Knot & Genus\\
 \hline
(1) & $T(3,6q+2;2,-1)$ & $6q$  \\
(2) & $T(3,6q+2;2,-2)$ & $6q-1$ \\
(3) & $T(3,6q+1;2,-1)$ & $6q-1$  \\
(4) & $T(3,6q-2;2,1)$ &  $6q-2$\\
(5) &  $T(3,6q-2;2,2)$ & $6q-1$ \\
(6) & $T(3,6q-1;2,1)$ & $6q-1$ \\
\end{tabular}
\caption{List of the pattern knot $P$ and its genus.}\label{table}
\end{table}

Recall that  $g(K_n)=9n+7$ and  $g(T)=(q-1)/2$.
If $g(P)=6q-1$, then $9n+7=3(q-1)/2+6q-1$, so $18n+14=3(q-1)+12q-2$.
Then $18n+14 \equiv -2 \pmod{3}$, a contradiction.
Thus (4) remains.
For this case, $9n+7=3(q-1)/2+6q-2$ gives $6n+7=5q$.
Then $n\equiv 3 \pmod{5}$.
Set $n=5m+3 \ (m\ge 0)$.  Then $q=6m+5$.

We have $\Delta_{K_n}(-1)=\Delta_T(-1)\Delta_P(-1)$, and
$\Delta_{K_n}(-1)=10n+11=50m+41$ from Theorem \ref{thm:alexander}.
However, $\Delta_T(-1)=q=6m+5$.
Then $6m+5$ does not divide $50m+41$, a contradiction.

Thus we have shown that our $K_n$ is hyperbolic.
\end{proof}

\begin{proof}[Proof of Theorem \textup{\ref{thm:main1}}] 
By Theorems \ref{thm:lspace} and \ref{thm:hyp},
$K_n$ is a hyperbolic L--space knot.
Its formal semigroup is described in  Theorem \ref{thm:formalsemigroup}.
\end{proof}


\section*{Acknowledgments}
The author would like to thank Ken Baker, Marc Kegel and Kimihiko Motegi for valuable communication, and
Yukinori Kitadai for his help of computer calculation.
The author also thanks the referee for valuable suggestions and comments.



\begin{thebibliography}{BKN}

\bibitem{A}
C. Anderson, K. Baker, X. Gao, M. Kegel, K. Le, K. Miller, S. Onaran, G. Sangston, S. Tripp, A. Wood and A. Wright,
L--space knots with tunnel number $>1$ by experiment,
preprint. \texttt{arXiv:1090.00790}.

\bibitem{BK}
K. Baker and M. Kegel,
Census L--space knots are braid positive, except for one that is not,
preprint. \texttt{arXiv:2203.12013}.

\bibitem{BKM}
K. Baker, M. Kegel and D. McCoy,
The search for alternating and quasi-alternating surgeries on asymmetric knots, preprint.

\bibitem{BMo}
K. Baker and A. Moore,
Montesinos knots, Hopf plumbings, and L--space surgeries,
{\it J. Math. Soc. Japan}.  {\bf 70} (2018) 95--110. 

\bibitem{BM}
K. Baker and K. Motegi,
Seifert vs. slice genera of knots in twist families and a characterization of braid axes,
{\it Proc. Lond. Math. Soc}. (3) {\bf 119} (2019) 1493--1530. 

\bibitem{BL0}
M. Borodzik and C. Livingston,
Heegaard Floer homology and rational cuspidal curves,
{\it Forum Math. Sigma}. {\bf 2} (2014), Paper No. e28, 23 pp. 

\bibitem{BL}
M. Borodzik and C. Livingston,
Semigroups, $d$--invariants and deformations of cuspidal singular points of plane curves,
{\it J. Lond. Math. Soc}. (2) {\bf 93} (2016) 439-463. 

\bibitem{BZ}
G. Burde, H. Zieschang and M. Heusener,
\textit{Knots}, Third, fully revised and extended edition,
(De Gruyter Studies in Mathematics, 5. De Gruyter, Berlin, 2014.)

\bibitem{F}
R. H. Fox, 
Free differential calculus II,
{\it Ann. of Math}. (2) {\bf 59} (1954) 196--210.


\bibitem{HRW}
J. Hanselman, J. Rasmussen and L. Watson,
Bordered Floer homology for manifolds with torus boundary via immersed curves,
preprint. \texttt{arXiv:1604.03466}.

\bibitem{HW}
M. Hedden and L. Watson, 
On the geography and botany of knot Floer homology,
{\it Selecta Math}. (N.S.) {\bf 24} (2018) 997--1037.

\bibitem{H}
E. Hironaka,
The Lehmer polynomial and pretzel links,
{\it  Canad. Math. Bull}. {\bf 44} (2001) 440--451.


\bibitem{HLV}
J. Hom, T. Lidman and F. Vafaee,
Berge--Gabai knots and L--space satellite operations,
{\it Algebr. Geom. Topol}. {\bf 14} (2014) 3745--3763.


\bibitem{I}
A. Issa,
The classification of quasi-alternating Montesinos links,
{\it Proc. Amer. Math. Soc}. {\bf 146} (2018) 4047--4057. 


\bibitem{Ko}
K. Kodama,
The software ``KNOT", a tool for knot theory, available at
\texttt{http://www.math.kobe-u.ac.jp/HOME/kodama/knot.html}


\bibitem{K}
D. Krcatovich,
A restriction on the Alexander polynomials of L--space knots,
{\it Pacific J. Math}. {\bf 297} (2018) 117--129.

\bibitem{LV}
C. R. S. Lee and F. Vafaee,
On $3$--braids and L--space knots,
{\it Geom. Dedicata}. {\bf 213} (2021) 513--521.

\bibitem{L}
S. Lee, 
Twisted torus knots that are unknotted,
{\it Int. Math. Res. Not}. (2014) 4958--4996. 

\bibitem{LM}
P. Lisca and G. Mati\'{c},
Transverse contact structures on Seifert $3$--manifolds,
{\it Algebr. Geom. Topol}. {\bf 4} (2004) 1125--1144.


\bibitem{LS}
P. Lisca and A. Stipsicz,
Ozsv\'{a}th-Szab\'{o} invariants and tight contact $3$--manifolds. III,
{\it J. Symplectic Geom}. {\bf 5} (2007) 357--384. 

\bibitem{Mon}
J. M. Montesinos,
Surgery on links and double branched covers of $S^3$, in 
{\it Knots, groups, and $3$--manifolds} (Papers dedicated to the memory of R. H. Fox), pp. 227--259.
(Ann. of Math. Studies, No. 84, Princeton Univ. Press, Princeton, N.J., 1975. )

\bibitem{MS}
K. Morimoto and M. Sakuma,
On unknotting tunnels for knots,
{\it Math. Ann}. {\bf 289} (1991) 143--167. 

\bibitem{M}
H. Morton,
The Alexander polynomial of a torus knot with twists,
{\it J. Knot Theory Ramifications}. {\bf 15} (2006) 1037--1047. 

\bibitem{MT}
K. Motegi and K. Tohki,
Hyperbolic L--space knots and exceptional Dehn surgeries,
{\it J. Knot Theory Ramifications}.  {\bf 23} (2014) 1450079, 13 pp. 


\bibitem{N}
Y. Ni,
Knot Floer homology detects fibred knots,
{\it Invent. Math}. {\bf 170} (2007) 577--608. 


\bibitem{OS}
P.  Ozsv\'{a}th and Z.  Szab\'{o},
On knot Floer homology and lens space surgeries,
{\it Topology}. {\bf 44} (2005) 1281--1300.

\bibitem{OS2}
P. Ozsv\'{a}th and Z. Szab\'{o},
On the Heegaard Floer homology of branched double-covers,
{\it Adv. Math}. {\bf 194} (2005) 1--33. 

\bibitem{OS3}
P. Ozsv\'{a}th and Z. Szab\'{o},
Knot Floer homology and rational surgeries,
{\it Algebr. Geom. Topol}. {\bf 11} (2011) 1--68. 

\bibitem{RR}
J. Rasmussen and S. Rasmussen,
Floer simple manifolds and L--space intervals,
{\it Adv. Math}. {\bf 322} (2017) 738--805.


\bibitem{Sc}
H. Schubert,
\"{U}ber eine numerische Knoteninvariante,
{\it Math. Z}. {\bf 61} (1954) 245--288.


\bibitem{S}
J.  Stallings, 
Constructions of fibred knots and links, in 
 {\it Algebraic and geometric topology}, {\it Proc. Sympos. Pure Math., Stanford Univ., Stanford, Calif., 1976}, 
 Part 2, pp. 55--60, (Proc. Sympos. Pure Math., XXXII, Amer. Math. Soc., Providence, R.I., 1978.)

\bibitem{Ta1}
M. Tange,
On the Alexander polynomial of lens space knots,
{\it Topology Appl}. {\bf 275} (2020) 107124, 37 pp.

\bibitem{Ta2}
M. Tange,
The third term in lens surgery polynomials,
{\it Hiroshima Math. J}. {\bf 51} (2021) 101--109.

\bibitem{T}
G. Torres,
On the Alexander polynomial,
{\it Ann. of Math}. (2) {\bf 57} (1953) 57--89. 


\bibitem{W}
S. Wang,
Semigroups of L--space knots and nonalgebraic iterated torus knots,
{\it Math. Res. Lett}. {\bf 25} (2018) 335--346. 

\bibitem{Wa}
L. Watson, 
A surgical perspective on quasi-alternating links, in 
{\it Low-dimensional and symplectic topology}, 39--51,
(Proc. Sympos. Pure Math., 82, Amer. Math. Soc., Providence, RI, 2011. )


\end{thebibliography}
\end{document}